\documentclass[12pt, reqno]{amsart}
\usepackage{amsmath, amsthm, amscd, amsfonts, amssymb, graphicx, color, mathrsfs}
\usepackage[bookmarksnumbered, colorlinks, plainpages]{hyperref}
\usepackage[all]{xy}
\usepackage{slashed}

\usepackage{soul}
\usepackage{cancel}
\usepackage{ulem}

\newtheorem{theorem}{Theorem}[section]

\theoremstyle{definition}

\theoremstyle{remark}
\newtheorem{remark}[theorem]{Remark}
\numberwithin{equation}{section}

\begin{document}
\setcounter{page}{1}

\title[  On the  nuclear trace  of Fourier integral  operators in $L^p$-spaces ]{ Nuclearity for Fourier integral  operators in $L^p$-spaces}

\author[D. Cardona]{Duv\'an Cardona}
\address{
  Duv\'an Cardona:
  \endgraf
  Department of Mathematics  
  \endgraf
  Pontificia Universidad Javeriana
  \endgraf
  Bogot\'a
  \endgraf
  Colombia
  \endgraf
  {\it E-mail address} {\rm d.cardona@uniandes.edu.co;
duvanc306@gmail.com}
  }



\begin{abstract}
In this note we study sharp sufficient conditions for the nuclearity of Fourier integral operators on $L^p$-spaces, $1< p\leq 2$.  Our conditions and those presented in Cardona \cite{FIO} provide a systematic investigation on the subject for all $1<p<\infty.$
\end{abstract} \maketitle

\section{Introduction}
In this note we are interested in those sharp conditions providing the nuclearity   of Fourier integral operators on Lebesgue spaces $L^p,$ for $1< p\leq 2$.  We complete our investigation on the subject for all $1<p<\infty,$ with this work, and those conditions proved in Cardona \cite{FIO} for $2\leq p<\infty.$   Fourier integral operators (FIOs) on $\mathbb{R}^n,$ are integral operators  of the form
\begin{equation}\label{definition}
Ff(x):=\int\limits_{\mathbb{R}^n}e^{i\phi(x,\xi)}a(x,\xi)\widehat{f}(\xi)d\xi,
\end{equation}
where $\mathscr{F}_{ }f:=\widehat{f}$ is the Fourier transform of  $f\in\mathscr{S}(\mathbb{R}^n).$ 
As it is well known, properties of FIOs on functions spaces   are used to 
study solutions to Cauchy problems (see H\"ormander \cite{Hor71,Hor1,Hor2} and Duistermaat and H\"ormander \cite{DuiHor}) for hyperbolic equations.  

By following  the theory of FIOs developed by H\"ormander \cite{Hor71}, the phase functions $\phi$  are positively homogeneous
of order 1 and they are considered smooth at $\xi\neq 0,$  while the  symbols are considered satisfying estimates of the form
\begin{equation}
\sup_{(x,y)\in K}|\partial_x^\beta\partial^\alpha_\xi a(x,y,\xi)|\leq C_{\alpha,\beta,K}(1+|\xi|)^{\kappa-|\alpha|},
\end{equation}
for every compact subset $K$ of $\mathbb{R}^{2n}.$  The action of Fourier integral operators on   $L^p$ spaces can be found in the references
H\"ormander \cite{Hor71}, Eskin\cite{Eskin},
Seeger, Sogge and Stein\cite{SSS91},
Tao\cite{Tao}, 
 Miyachi \cite{Miyachi}, Peral\cite{Peral}, Asada and Fujiwara\cite{AF}, Fujiwara\cite{Fuji}, Kumano-go\cite{Kumano-go}, Coriasco and Ruzhansky \cite{CoRu1,CoRu2}, Ruzhansky and Sugimoto \cite{RuzSugi01,RuzSugi02,RuzSugi03,RuzSugi}, Ruzhansky \cite{M. Ruzhansky}, and Ruzhansky and Wirth \cite{RWirth} where the problem has been extensively investigated.

In this note our main goal  is to provide sufficient conditions  for the $r$-nuclearity of Fourier integral operators on $L^p$-spaces. This problem has been considered in the case of pseudo-differential operators  by several authors and for FIOs in Cardona \cite{FIO}. Now, we present a summary of the works on the subject. Sufficient conditions guarantying  the nuclearity  of pseudo-differential operators on  $\mathbb{S}^1,\mathbb{Z},$ arbitrary compact Lie groups and   (closed) compact manifolds can be found  in the works of Delgado, Ruzhansky, Wong \cite{DR,DR1,DR3,DR5,DRboundedvariable2,DW} and Cardona \cite{Cardona}; similar conditions on different functions spaces  can be found  in the works Delgado and Ruzhansky\cite{DRboundedvariable,DRB,DRB2}.  Finally, the subject was treated for  compact manifolds with boundary by Delgado, Ruzhansky, and Tokmagambetov in \cite{DRTk}. 

Our work is motivated by the recent works of Ghaemi,  Jamalpour Birgani, and  Wong in \cite{Ghaemi,Ghaemi2,Majid} for $\mathbb{S}^1,\mathbb{Z}$ and also for arbitrary compact and Hausdorff groups. In these references the authors have characterized the nuclearity of pseudo-differential operators by showing that symbols associated to nuclear pseudo-differential operators admit a suitable decomposition where the spatial and momentum variables appear separately. In the compact case, the main tool for providing these characterizations is the fact that the unitary dual of compact and Hausdorff groups is a discrete set. This situation is different for the case of operators on $\mathbb{R}^n$ because the unitary dual of $\widehat{\mathbb{R}}^n\equiv\{ e^{i2\pi x\cdot\xi}:\xi\in\mathbb{R}^n \}  $ is merely continuous. With this in mind, the techniques used in the euclidean case, could be slightly different or as in the proof of our main theorem far from of those used in  \cite{Ghaemi,Ghaemi2,Majid}.

In order to present our main result we recall of notion of $r$-nuclearity of Grothendieck.  By following  \cite{GRO}, a densely defined  linear operator $T:D(T)\subset E\rightarrow F$  (where $D(T)$ is the domain of $T,$ and $E,F$ are  Banach spaces) extends to a  $r$-nuclear operator from $E$ into $F$, if
there exist  sequences $(e_n ')_{n\in\mathbb{N}_0}$ in $ E'$ (the dual space of $E$) and $(y_n)_{n\in\mathbb{N}_0}$ in $F$ such that, the discrete representation
\begin{equation}\label{nuc}
Tf=\sum_{n\in\mathbb{N}_0} \langle e_n',f\rangle y_n,\,\,\, \textnormal{ with }\,\,\,\sum_{n\in\mathbb{N}_0} \Vert e_n' \Vert^r_{E'}\Vert y_n \Vert^r_{F}<\infty,
\end{equation} holds true for all $f\in D(T).$
\noindent The class of $r-$nuclear operators is usually endowed with the natural semi-norm
\begin{equation}
n_r(T):=\inf\left\{ \left\{\sum_n \Vert e_n' \Vert^r_{E'}\Vert y_n \Vert^r_{F}\right\}^{\frac{1}{r}}: T=\sum_n e_n'\otimes y_n \right\}.
\end{equation}
If $r=1$, $n_1(\cdot)$ is a norm and we obtain the ideal of nuclear operators. In addition, when $E=F$ is a Hilbert space $H$ and  $r=1$ the definition above agrees with that of   trace class operators. For the case of Hilbert spaces $H$, the set of $r$-nuclear operators agrees with the Schatten-von Neumann class of order $r$ (see Pietsch  \cite{P,P2}).

An important notion associated to the theory of $r$-nuclear operators is that of trace.
If we choose a $r$-nuclear  operator  $T:E\rightarrow E$, $0<r\leq 1,$ with the Banach space $E$ satisfying the Grothendieck approximation property (see Grothendieck\cite{GRO}),  then 
 the nuclear trace of $T$ is (a  well-defined functional) given by
$$
\textnormal{Tr}(T)=\sum_{n\in\mathbb{N}^n_0}e_n'(f_n).
$$

In order to illustrate our results about the $r$-nuclearity and the nuclear trace of Fourier integral operators, allow us to recall the following criterion, (see Cardona \cite{FIO}). Throughout this document the phase function $\phi$ will be considered real valued and measurable.
\begin{theorem}
 Let  $0<r\leq 1.$ Let $a(\cdot,\cdot)$ be a symbol such that $a(x,\cdot)\in L^1_{loc}(\mathbb{R}^n),$ $a.e.w.,$  $x\in \mathbb{R}^n.$ Let $2\leq p_1<\infty,$   $1\leq p_2<\infty,$  and let $F$ be the Fourier integral operator associated to $a(\cdot,\cdot).$ Then, $F:L^{p_1}(\mathbb{R}^n)\rightarrow L^{p_2}(\mathbb{R}^n)$ is $r$-nuclear, if and only if,  the symbol $a(\cdot,\cdot)$  admits a decomposition of the form
\begin{equation}
a(x,\xi)=\frac{1}{e^{i\phi(x,\xi)}}\sum_{k=1}^{\infty}h_{k}(x)(\mathscr{F}^{-1}{g}_k)(\xi),\,\,\,a.e.w.,\,\,(x,\xi),
\end{equation} where   $\{g_k\}_{k\in\mathbb{N}}$ and $\{h_k\}_{k\in\mathbb{N}}$ are sequences of functions satisfying 
\begin{equation}
\mathbb{E}^r(g,f):=\sum_{k=0}^{\infty}\Vert g_k\Vert^r_{L^{p_1'}}\Vert h_{k}\Vert^r_{L^{p_2}}<\infty.
\end{equation} In this case, the nuclear trace of $F$ is given by
$$\textnormal{Tr}(F)=\int\limits_{\mathbb{R}^n}\int\limits_{\mathbb{R}^n}   e^{i\phi(x,\xi)-2\pi ix\cdot \xi}a(x,\xi)d\xi\,dx.$$
\end{theorem}

As a complement of the previous result we present our main theorem where we consider the $r$-nuclearity of $F$ from $L^{p_1}$ into $L^{p_2}$ for $1<p_1\leq 2.$

\begin{theorem}\label{MainTheorem}
 Let  $0<r\leq 1.$ Let $a(\cdot,\cdot)$ be a measurable on  $ \mathbb{R}^{2n}.$ Let $1<p_1\leq 2,$   $1\leq p_2<\infty,$ and  $F$ be the Fourier integral operator associated to $a(\cdot,\cdot).$ Then, $F:L^{p_1}(\mathbb{R}^n)\rightarrow L^{p_2}(\mathbb{R}^n)$ is $r$-nuclear if  the symbol $a(\cdot,\cdot)$  admits a decomposition of the form
\begin{equation}\label{symboldecompositionT}
a(x,\xi)=\frac{1}{e^{i\phi(x,\xi)}}\sum_{k=1}^{\infty}h_{k}(x){g}_k(\xi),\,\,\,a.e.w.,\,\,(x,\xi),
\end{equation} where   $\{g_k\}_{k\in\mathbb{N}}$ and $\{h_k\}_{k\in\mathbb{N}}$ are sequences of functions satisfying 
\begin{equation}\label{cond1}
\mathbb{E}^r(g,f):=\sum_{k=0}^{\infty}\Vert g_k\Vert^r_{L^{p_1}}\Vert h_{k}\Vert^r_{L^{p_2}}<\infty.
\end{equation} This theorem is sharp in the sense that the previous condition is a necessary and sufficient condition for the $r$-nuclearity of $F$ when $p_1=2.$ 
\end{theorem}

\begin{remark} Note that from Reinov and Latif \cite{O}, the nuclear trace of every $r$-nuclear operator on $L^p(\mathbb{R}^n),$ agrees with its spectral trace provided that $1/r=1+|1/p-1/2|.$ If $F$ is a nuclear operator on $L^2(\mathbb{R}^n)$ we have, $$\textnormal{Tr}(F)=\int\limits_{\mathbb{R}^n}\int\limits_{\mathbb{R}^n}e^{i\phi(x,\eta)-i2\pi x\cdot\eta}a(x,\eta)d\eta d\xi =\sum_{n=0}^\infty\lambda_n(F),$$ where $\lambda_n(F),$ $n\in\mathbb{N}_0,$ is the sequence of eigenvalues of $F$ (counting multiplicity).
\end{remark}

\section{Proof of the main result}\label{Rnn}
In this section we prove our main result for Fourier integral operators $F$ defined as in \eqref{definition}. First, let us observe that every FIO $F$ has a integral representation with kernel l $K(x,y).$ In fact, straightforward computation shows us that
\begin{equation}\label{kernelpseudo}
Ff(x) =\int\limits_{\mathbb{R}^n}K(x,y)f(y)dy,
 \,\, K(x,y):=\int_{\mathbb{R}^n}e^{i\phi(x,\xi)-i2\pi y\cdot \xi}a(x,\xi)d\xi, \end{equation}
for every  $f\in\mathscr{D}(\mathbb{R}^n).$ In order to analyze the $r$-nuclearity of the Fourier integral operator  $F$ we study  its kernel $K,$ by using as fundamental tool, the following theorem (see J. Delgado \cite{Delgado,D2}).
\begin{theorem}[Delgado Theorem]\label{Theorem1} Let us consider $1\leq p_1,p_2<\infty,$ $0<r\leq 1$ and let $p_i'$ be such that $\frac{1}{p_i}+\frac{1}{p_i'}=1.$ Let $(X_1,\mu_1)$ and $(X_2,\mu_2)$ be $\sigma$-finite measure spaces. An operator $T:L^{p_1}(X_1,\mu_1)\rightarrow L^{p_2}(X_2,\mu_2)$ is $r$-nuclear if and only if there exist sequences $(h_k)_k$ in $L^{p_2}(\mu_2),$ and $(g_k)$ in $L^{p_1'}(\mu_1),$ such that
\begin{equation}
\sum_{k}\Vert h_k\Vert_{L^{p_2}}^r\Vert g_k\Vert^r_{L^{p_1'} } <\infty,\textnormal{        and        }Tf(x)=\int\limits_{X_1}(\sum_k h_{k}(x)g_k(y))f(y)d\mu_1(y),\textnormal{   a.e.w. }x,
\end{equation}
for every $f\in {L^{p_1}}(\mu_1).$ In this case, if $p_1=p_2,$ and $\mu_1=\mu_2,$ $($\textnormal{see Section 3 of} \cite{Delgado}$)$ the nuclear trace of $T$ is given by
\begin{equation}\label{trace1}
\textnormal{Tr}(T):=\int\limits_{X_1}\sum_{k}g_{k}(x)h_{k}(x)d\mu_1(x).
\end{equation}
\end{theorem}

Now, curiously, we present the novelty of this paper that is the short proof of our main result. 

\begin{proof}[Proof of Theorem \ref{MainTheorem}]

Let us consider the Fourier integral operator $F,$ 
\begin{equation}
Ff(x):=\int\limits_{\mathbb{R}^n}e^{i\phi(x,\xi)}a(x,\xi)\widehat{f}(\xi)d\xi,
\end{equation}
with associated symbol $a.$ The main strategy for  the proof will be to analyze the natural factorization of $F$ in terms of the Fourier transform,
\begin{equation}
    (\mathscr{F}f)(\xi):=\int\limits_{\mathbb{R}^n}e^{-i2\pi x\cdot \xi}f(x)dx.
\end{equation} 
Clearly, if  we define the operator with kernel (associated to $\sigma=(\phi,a)$),
\begin{equation}
K_\sigma g(x)=\int\limits_{\mathbb{R}^n}e^{i\phi(x,\xi)}a(x,\xi)g(\xi),\,\,g\in\mathscr{S}(\mathbb{R}^n),\,\,\,K_\sigma(x,\xi)=e^{i\phi(x,\xi)}a(x,\xi),
\end{equation} then $
    F=K_\sigma \circ \mathscr{F}.$
Taking into account the Hausdorff-Young inequality, 
\begin{equation}
\Vert \mathscr{F}f\Vert_{L^{p_1'}(\mathbb{R}^n)}\leq \Vert f\Vert_{L^{p_1} (\mathbb{R}^n) },
\end{equation} 
 the Fourier transform extends to a bounded operator from $L^{p_1}(\mathbb{R}^n)$ into $L^{p_1'}(\mathbb{R}^n).$  So, if we prove that the condition \eqref{cond1} assures the $r$-nuclearity of $K_\sigma$ from $L^{p_1'}(\mathbb{R}^n) $ into $L^{p_2}(\mathbb{R}^n) ,$ we can deduce the $r$-nuclearity of $F$ from $L^{p_1}(\mathbb{R}^n) $ into $L^{p_2}(\mathbb{R}^n) .$  Here, we will be using that the class of $r$-nuclear operators is a bilateral ideal on the set of bounded operators between Banach spaces.
 
 Now, from Theorem \ref{Theorem1}, $K_\sigma:L^{p_1'}(\mathbb{R}^n)\rightarrow L^{p_2}(\mathbb{R}^n) $ is $r$-nuclear, if and only if, there exist sequences $\{h_k\},\{g_k\}$ satisfying
 \begin{equation}\label{1111}
    K_\sigma(x,\xi)= e^{i\phi(x,\xi)}a(x,\xi)=\sum_{k}h_k(x)g_{k}(\xi),
 \end{equation} where \begin{equation}
\sum_{k}\Vert h_k\Vert_{L^{p_2}}^r\Vert g_k\Vert^r_{L^{p_1} } <\infty,\textnormal{        and        }K_\sigma f(x)=\int\limits_{\mathbb{R}^n}(\sum_k h_{k}(x)g_k(\xi))g(\xi)d\xi,\textnormal{   a.e.w. }x,
\end{equation}
for every $g\in {L^{p_1'}}(\mathbb{R}^n).$ Here, we have used that for $1<p_1\leq 2,$  $L^{p_1''}(\mathbb{R}^n)=L^{p_1}(\mathbb{R}^n).$ We end the proof by observing that \eqref{1111} is in turns   equivalent  to \eqref{symboldecompositionT}.
\end{proof}   
\begin{remark}[Sharpness of Theorem \ref{Theorem1}] Let us note, that from Plancherel Theorem, $\mathscr{F}:L^2(\mathbb{R}^n)\rightarrow L^2(\mathbb{R}^n)$ extends to an isomorphism of Hilbert spaces. In the proof of Theorem \ref{Theorem1}, $\mathscr{F}$ is an invertible and bounded operator for $p_1=2.$ Hence, from the relation $F=K_\sigma\circ \mathscr{F},$ we deduce that $F$ is $r$-nuclear if and only if $K_\sigma=F\circ \mathscr{F}^{-1}$ is also $r$-nuclear.
But, $K_\sigma$ is $r$-nuclear from $L^{2}$ into $L^2$  if and only if \eqref{symboldecompositionT} holds true. So, for $p_1=2,$ \eqref{symboldecompositionT} characterizes the $r$-nuclearity of $F$ by showing that it is a necessary and sufficient condition for this fact.
\end{remark}

\bibliographystyle{amsplain}

\end{document}